\documentclass
{amsart}
\usepackage{
amsmath, 
amsfonts, 
amssymb, 
amsthm, 
amscd
}

\usepackage[all]{xy}
\usepackage{enumerate}

\newtheorem{theorem}{Theorem}[section]
\newtheorem{lemma}[theorem]{Lemma}
\newtheorem{corollary}[theorem]{Corollary}

\newtheorem{prop}[theorem]{Proposition}

\newtheorem{conjecture}[theorem]{Conjecture}

\theoremstyle{definition}
\newtheorem{defn}[theorem]{{Definition}}
\newtheorem{example}[theorem]{Example}
\newtheorem{chunk}[theorem]{\hspace*{-1.065ex}\bf}

\newtheorem*{chunk*}{}
\newtheorem*{ack}{Acknowledgement}
\numberwithin{equation}{section}

\theoremstyle{remark}
\newtheorem{remark}[theorem]{Remark}

\input xy
\xyoption{all}

\usepackage[all]{xy}
\SelectTips{cm}{} 

\newcommand{\xra}[1]{\xrightarrow{#1}}

\newcommand{\ch}{ch}
\newcommand{\del}{\partial}

\newcommand{\et}{{\operatorname{\acute{e}t}}}

\newcommand{\ev}{\operatorname{ev}}
\newcommand{\fm}{\mathfrak{m}}
\newcommand{\fp}{\mathfrak{p}}

\renewcommand{\H}{{\operatorname{H}}}

\newcommand{\iso}{\cong}
\newcommand{\K}[1]{K(#1)}
\newcommand{\G}[1]{G(#1)}
\newcommand{\sk}{\mathsf{k}}
\newcommand{\length}{\operatorname{length}}
\newcommand{\cO}{\mathcal{O}}
\newcommand{\odd}{\operatorname{odd}}

\newcommand{\bP}{\mathbb{P}}
\newcommand{\Proj}[1]{\operatorname{Proj}(#1)}
\newcommand{\myProj}{\operatorname{Proj}}
\newcommand{\bQ}{\mathbb{Q}}
\newcommand{\bC}{\mathbb{C}}

\newcommand{\Tor}{\operatorname{Tor}}

\newcommand{\bZ}{\mathbb{Z}}

\newcommand{\Spec}{\operatorname{Spec}}

\newcommand{\Hs}[2][]{\operatorname{H}_{#1}(#2)}

\def\Het{\operatorname{H}_{\text{\'et}}}

\newcommand{\ds}{\displaystyle}

\begin{document}


\title[An invariant for complete intersections]
{The vanishing of a higher codimension analogue of Hochster's theta invariant}

\author{W. Frank Moore} \address{Department of Mathematics, Wake Forest
  University, Winston-Salem, NC 27109} \email{moorewf@wfu.edu}

\author{Greg Piepmeyer} \address{University
  of Missouri, Columbia, MO 65211} \email{ggpz95@mail.missouri.edu}

\author{Sandra Spiroff} \address{Department of Mathematics, University
  of Mississippi, University, MS 38677} \email{spiroff@olemiss.edu}

\author{Mark E. Walker} \address{Department of Mathematics, University
  of Nebraska, Lincoln, NE 68588} \email{mwalker5@math.unl.edu}
\thanks{The fourth author was supported in part by NSF grant
  DMS-0601666}

\begin{abstract}
  We study H.~Dao's invariant $\eta_c^R$ of pairs of modules defined
  over a complete intersection ring $R$ of codimension $c$ having an
  isolated singularity.  Our main result is that $\eta_c^R$ vanishes
  for all pairs of modules when $R$ is a {\em graded} complete
  intersection ring of codimension $c > 1$ having an isolated
  singularity. A consequence of this result is that all pairs of
  modules over such a ring are $c$-$\Tor$-rigid.
\end{abstract}

\date \today
\maketitle


\section{Introduction}


Let $R$ be an isolated complete intersection singularity, i.e., $R$
is the quotient of a regular local ring $(Q, \fm)$ by a regular
sequence $f_1, \dots, f_c \in Q$, and $R_\fp$ is regular for all $\fp
\ne \fm$.  For any pair $(M,N)$ of finitely generated $R$-modules, the
$\Tor$ modules $\Tor^R_j(M,N)$ have finite length when $j \gg
0$. Moreover, the lengths of the odd and even indexed $\Tor$ modules
in high degree follow predictable patterns.  There are polynomials
$$
P_{\ev}(j) = a_{c-1}j^{c-1} + \cdots +a_1j + a_0 \, \text{ and } \, 
P_{\odd}(j) = b_{c-1}j^{c-1} +
\cdots + b_1j + b_0
$$ of degree at most $c-1$ such that
$$
\length \Tor^R_{2j}(M,N) = P_{\ev}(j)
\, \text{ and } \,
\length \Tor^R_{2j+1}(M,N) = P_{\odd}(j) 
$$
for all $j \gg 0$ (see, for example,  \cite[Theorem 4.1]{Dao3}).

The polynomials $P_{\ev}$ and $P_{\odd}$ need not be the same, nor is
it necessary that the coefficients of $j^{c-1}$ coincide.  A natural
invariant of the pair $(M,N)$ is thus the difference, $a_{c-1} -
b_{c-1}$, of these coefficients. Up to a constant factor, this
difference is {\em Dao's $\eta_c^R$-invariant}.

For example, if $c=1$, then $P_{\ev} = a_0$ and $P_{\odd} = b_0$ are
both constants.  This reflects the fact that, in our present context,
the Eisenbud operator
$$
\chi:~\Tor^R_j(M,N) \xrightarrow{\cong} \Tor^R_{j-2}(M,N)
$$
is an isomorphism for $j \gg 0$.  In this case, the invariant
$\eta_1^R(M,N)$ is (one half of) the difference $a_0 - b_0$ of these
$\Tor$-lengths. This difference is {\em Hochster's
  $\theta$-invariant}.

In our previous paper \cite{MPSW} we studied Hochster's
$\theta$-invariant in the special case where $R$ is a graded, isolated
hypersurface singularity. Now we employ similar techniques to study
the invariant $\eta_c^R$ for graded, isolated complete intersection
singularities of codimension $c > 1$. That is, we assume
\begin{equation} \label{E29}
R = \sk[x_0, \dots, x_{n+c-1}]/(f_1, \dots, f_c)
\end{equation}
where $\sk$ is a field, the $f_l$'s are homogeneous polynomials, and
$\myProj(R)$ is a smooth $\sk$-variety.  With these assumptions, the
irrelevant maximal ideal $\fm = (x_0, \dots, x_{n+c-1})$ of $R$ is the
only non-regular prime in $\Spec R$.  Our main result is that, for
such a ring $R$, the invariant $\eta_c^R(M,N)$ vanishes for all
$R$-modules $M$ and $N$, provided $c > 1$.  See Theorem
\ref{etacvanish}.

H.~Dao \cite[Theorem 6.3]{Dao3} has proven that the vanishing of
$\eta^R_c(M,N)$ implies that the pair $(M,N)$ is {\em
  $c$-$\Tor$-rigid}, meaning that if $c$ consecutive $\Tor$ modules
vanish, then all subsequent $\Tor$ modules vanish too.  We therefore
conclude that all pairs of modules over rings of the form \eqref{E29}
having an isolated singularity are $c$-$\Tor$-rigid, provided $c > 1$;
our previous result \cite[Remark 3.16]{MPSW} shows $c$-$\Tor$-rigidity
for $c = 1$ when the dimension of $R$ is even.  In general, one only
has $(c+1)$-$\Tor$-rigidity for pairs of modules over a codimension
$c$ complete intersection \cite{Mur}.

We conjecture that $\eta^R_c(M,N) = 0$ for all pairs of modules
$(M,N)$ over an isolated complete intersection singularity of
codimension $c > 1$, and hence that all pairs of modules over such a
ring are $c$-$\Tor$-rigid.  There are many well-known examples (of
isolated hypersurface singularities) where $(M,N)$ is not
$1$-$\Tor$-rigid, and hence $\theta^R(M,N) \ne 0$.

The invariant $\eta^R_c(M,N)$ is also defined for complete
intersection rings that are not isolated singularities, provided the
pair $(M,N)$ has the property that $\Tor^R_j(M,N)$ has finite length
for all $j \gg 0$.  We include an example due to D.~Jorgensen and
O.~Celikbas that shows that $\eta^R_2$ need not vanish for a complete
intersection of codimension $2$ if the dimension of the singular locus
is positive.


\section{Dao's $\eta_c^R$-Invariant}


In this section we recall the definition of Dao's $\eta_c^R$-invariant
for complete intersections.

\begin{prop} \label{P's} Let $R$ be the quotient of a noetherian ring
  $Q$ by a regular sequence $f_1, \dots, f_c$.  For a pair of finitely
  generated $R$-modules $M$ and $N$, suppose the $Q$-module
  $\Tor^Q_j(M,N)$ vanishes for all $j \gg 0$ and that the $R$-modules
  $\Tor^R_j(M,N)$ are supported on a finite set of maximal ideals
  $\{\fm_1, \dots, \fm_s\}$ of $R$ for all $j \gg 0$.  Then there are
  polynomials $P_{\ev}= P_{\ev}^R(M,N)$ and $P_{\odd} =
  P_{\odd}^R(M,N)$ of degree at most $c-1$ so that
\begin{equation*}
\length \Tor_{2j}^R(M,N) = P_{\ev}(j)  
\, \text{ and } \,
\length \Tor_{2j+1}^R(M,N) = P_{\odd}(j)  
\end{equation*} 
for all $j \gg 0$. 
\end{prop} 

\begin{proof}
  Apply \cite[Theorem 4.1(2)]{Dao3} to each $R_{\fm_i}$ and add the resulting
  polynomials to obtain the polynomials here. See Appendix
  \ref{App} for an alternative proof of this result. \end{proof}

The difference of the coefficients of $j^{c-1}$ in
$P_{\ev}=P_{\ev}^R(M,N)$ and in $P_{\odd}$ is the basis for an
invariant of $(M, N)$.  We can obtain these coefficients through the
$(c-1)$-st iterated first difference: the first difference
of a polynomial $q(j)$ is the polynomial $q^{(1)}(j) = q(j) - q(j-1)$,
and recursively one defines $ q^{(i)} = (q^{(i-1)})^{(1)}$.

\begin{defn}
In the set up of Proposition \ref{P's}, define
$$\eta_c^R(M,N) = \frac{(P_{\ev} - P_{\odd})^{(c-1)}}{2^c \cdot c!}.$$
\end{defn}

This invariant of the pair $(M,N)$ is Dao's $\eta_c^R$-invariant
\cite[4.2]{Dao3}.
\begin{remark} \label{eta&theta} %
  For a pair of $R$-modules, Dao sets 
\begin{equation*}
\begin{aligned}
 \beta_j(M,N) & = \begin{cases} \length{\Tor^R_j(M, N)} & \text{if
     $\length{\Tor^R_j(M, N)} < \infty$ and} \\
                                      0                       &
                                      \text{otherwise.} \end{cases} \\
\end{aligned}
\end{equation*}
It can be shown, as an easy application of Proposition \ref{P's}, that
under the assumptions in that result, if $c > 0$, then
\begin{equation*}
  \eta_c^R(M,N) = \lim_{n \to \infty} \frac{\sum_{j = 0}^n (-1)^j
    \beta_j(M,N)}{n^c}. 
\end{equation*}
This limit is the original definition of $\eta_c^R(M,N)$ due to Dao.
\end{remark}

Our main result, Theorem \ref{etacvanish}, suggests
the following conjecture:

\begin{conjecture} \label{conj} %
  Suppose $R = Q/(f_1, \dots, f_c)$ with $Q$ a regular noetherian ring
  and $f_1, \dots, f_c$ a regular sequence, with $c > 1$. If the
  singular locus of $R$ consists of a finite number of maximal ideals,
  then $\eta_c^R(M,N) = 0$ for all finitely generated $R$-modules $M$
  and $N$.
\end{conjecture}

\begin{remark}
  The case $N=\sk$ of Conjecture \ref{conj} follows from a result due
  to L.~Avramov, V.~Gasharov, and I.~Peeva \cite[8.1]{AGP}.  In this
  case, the length of the $\Tor_j$ record the Betti numbers of $M$
  over $R$, and part of their result states that both the even and odd
  Betti numbers grow at the same polynomial rate and have the same
  leading coefficient.
\end{remark}

\begin{example} \label{CJex} %
  The following example is due to D.~Jorgensen \cite[Example 4.1]{Jor}
  and O.~Celikbas \cite[Example 3.11]{Cel}.  Let $\sk$ be a field, and
  let $R = \sk[\![x,y,z,u]\!]/(xy,zu)$.  Then $R$ is a local complete
  intersection of codimension two with positive dimensional singular
  locus.  Let $M = R/(y,u)$, and let $N$ be the cokernel of the map
$$
R^2 \xrightarrow{\left(\begin{matrix}0 & u \\ -z & x \\ y &
      0\end{matrix}\right)} R^3.
$$
Then the pair $(M,N)$ is not $2$-$\Tor$-rigid, and hence, by
\cite[6.3]{Dao3}, $\eta_2^R(M,N) \neq 0$.
\end{example}


\section{The Graded Case}


\begin{chunk} \label{opening assumptions}
Let $Q$ be a graded noetherian ring, $f_1,\dots,f_c$ be a $Q$-regular
sequence of homogeneous elements, and $R = Q/(f_1,\dots,f_c)$. Then
for each pair of finitely generated graded $R$-modules $M$ and $N$,
$\Tor_j^R(M,N)$ is a graded $R$-module for all $j$.  Moreover, with
the notation $d_l = \deg f_l$, the Eisenbud operators \cite{Ei} $\chi_1,
\dots, \chi_c$ determine maps of graded $R$-modules
\begin{equation*}
  \chi_l \colon \Tor_j^R(M,N) \to \Tor_{j-2}^R(M,N) (-d_l)  
\end{equation*} 
for all $j$, where for a graded $R$-module $T$, we define $T(m)$ to be
the graded $R$-module satisfying $T(m)_k = T_{k + m}$.  Since the
actions of the $\chi_l$ commute, we may view $\bigoplus_{j,i}
\Tor_j^R(M,N)_i^{}$ as a bigraded module over the bigraded ring $S =
R[\chi_1,\dots,\chi_c]$, where the degree of $\chi_l$ is $(-2,-d_l)$.
\end{chunk}

The operators $\chi_l$ first appeared in work of Gulliksen \cite{Gu}
as (co)homology operators (albeit in a different guise), where he
proved that $\Tor_*^R(M,N)$ is artinian over $S$ if and only if
$\Tor_*^Q(M,N)$ is artinian over $Q$.  Compare our Appendix \ref{App}.  

\begin{prop} \label{prop1} %
  Let $Q$ and $R$ be as in paragraph \ref{opening assumptions}.  For a pair
  of finitely generated graded $R$-modules $M$ and $N$, suppose the
  $Q$-module $\Tor^Q_j(M,N)$ vanishes for all $j \gg 0$ and there is a
  finite set of maximal ideals of $R$ on which the $R$-modules
  $\Tor^R_j(M,N)$ are supported for all $j \gg 0$.  Then the action of
  the Eisenbud operators induces an
  exact (Koszul) sequence
\begin{multline} %
\label{GradedKoszul} %
   0 \to \Tor_j^R(M,N)\to \bigoplus_{l=1}^c  \Tor_{j-2}^R(M,N)(-d_l) \to 
   \\  \bigoplus_{l_1 < l_2} \Tor_{j-4}^R(M,N)(-d_{l_1} - d_{l_2}) \to
    \cdots  \to \Tor_{j-2c}^R(M,N)(-d_1 - \cdots -d_c)
   \to 0
\end{multline}
of graded $R$-modules for $j \gg 0$.
\end{prop} 

\begin{proof} %
  Within this proof, we use a simplified grading on $S = R[\chi_1,
  \ldots, \chi_c]$, effectively ignoring the twists given by the $d_l$
  in (\ref{GradedKoszul}).  We let $R$ lie in degree $0$ and let each
  $\chi_l$ lie in degree $-2$.  By \cite[Lemma 3.2]{Dao3}, for a
  sufficiently large $J$, the module $T = \bigoplus_{j \geq J}
  \Tor^R_j(M,N)$ is graded artinian over the ring $S$.  Consider the
  Koszul complex $K = K[\chi_1, \ldots, \chi_c] \otimes_S T$.  For $j \gg
  0$, the complex (\ref{GradedKoszul}) is the $j$th graded piece of
  $K$.  We prove this complex $K$ is exact in all but finitely many
  degrees.

  As $T$ is graded artinian over $S$, the total homology module
  $H(K)$ is as well.  The descending chain of $R$-submodules
\begin{equation*} 
H(K) \supseteq \bigoplus_{j \geq 1} H(K)_j
\supseteq\bigoplus_{j \geq 2} H(K)_j \supseteq \cdots
\end{equation*}  
intersects to $0$.  Since $\chi_l H(K) = 0$, see
\cite[IV.A.4]{Ser}, these $R$-submodules are in fact
$S$-submodules.  Since $H(K)$ is artinian over $S$, the descending
chain stabilizes.  Thus there exists an $m>0$ such that $H(K)_j=0$ for
all $j \geq m$. 
\end{proof}

\begin{chunk} \label{intermediate assumptions} For the remainder of
  this section, we assume $Q = \sk[x_0, \ldots, x_{n+c-1}]$ is a
  polynomial ring over a field with each $x_i$ of degree one, $f_1,
  \dots, f_c$ is a $Q$-regular sequence of homogeneous elements, and
  $R = Q/(f_1,\dots,f_c)$.  Let $d_l = \deg(f_l)$.  In particular, $R$
  is graded.  When $M$ and $N$ are finitely generated graded
  $R$-modules, the torsion modules $\Tor^R_j(M,N)$ are also graded.
\end{chunk}

\begin{defn} \label{Gdef} %
  Let $R$ be as in paragraph \ref{intermediate assumptions}.  For
  finitely generated graded $R$-modules $M$ and $N$, and an integer
  $F$, define
$$
G_F(x,t) = \sum_{i,j\geq 0} \dim_\sk \left(\Tor^R_{F+2j}(M,N)_i \right)
t^i x^j \in
\bQ[[x,t]].
$$
\end{defn}
\begin{remark} \label{fin len M N}
Note that if for some $F \gg 0$, $\Tor_{F + 2j}^R(M,N)$ has finite
length for all $j \geq 0$, then $G_F(x,t)$ belongs to $(\bQ[t])[[x]]$.
\end{remark}

For a finitely generated graded $R$-module $T$, its {\emph{Hilbert
    series}} is
$$
H_T(t) = \sum_{i \geq 0} \dim_{\sk}(T_i) t^i.
$$
$H_T(t)$ is a rational function with a pole of order $\dim{T}$ at
$t=1$.  In fact, 
\begin{equation} \label{neweq} 
  H_T(t) = \frac{e_T(t)}{(1-t)^{\dim T}}, 
\end{equation}
where $e_T(t)$ is a Laurent polynomial \cite[(1.1)]{AB}, sometimes
called the {\it multiplicity polynomial} of $T$.  The multiplicity
polynomial of $R$ is calculated by using the presentation $R = Q/(f_1,
\ldots, f_c)$; explicitly, since $\deg(x_i) = 1$ and $\deg(f_l) =
d_l$, 
\begin{equation} \label{e_R(t)}
e_R(t)=\prod_{l=1}^c (1- t^{d_l})/(1 - t)^c.
\end{equation}

For graded $R$-modules $M$ and $N$, %
we write the Hilbert series of $\Tor_j^R(M,N)$ as $H_j(t)$ or just
$H_j$.  If $M$ and $N$ are such that the $\Tor_j^R(M, N)$ have finite
length for $j \gg 0$, then $H_j(t)$, $j \geq 0$, has the property that
the number of initial terms of $H_j(t)$ that vanish goes to infinity
as $j \to \infty$.  It thus makes sense to form the sum
$$
\sum_{j \geq 0} (-1)^j H_j(t),
$$
and, more generally, to evaluate $G_F(x,t)$ at $x=1$ (or any
constant; see Remark \ref{fin len M N}).  Observe that
\begin{equation} \label{alternating}
\sum_{j \geq 0} (-1)^j H_{j+F}(t) = G_F(1,t) - G_{F+1}(1,t),
\end{equation} 
for any $F$.
   
\begin{lemma} \label{polynomials} %
  Let $R$ be a graded ring as in paragraph \ref{intermediate
    assumptions}, and let $M$, $N$ be finitely generated graded
  $R$-modules such that $\Tor_j^R(M, N)$ has finite length for $j \gg
  0$.  For $F \gg 0$ there is a unique %
  polynomial $b_F(x,t) \in \bQ[x, t]$ such that
\begin{equation} \label{Gasratfun}
  G_F(x,t) = \frac {b_F(x,t)}{\prod_{l=1}^c (1 - t^{d_l}x)}.
\end{equation} 
For $E \gg 0$ and even there is a unique %
polynomial $\eta_{c, E}^R(M, N)(t) \in \bQ[t]$ such that
\begin{equation} \label{etacrelation}
  \sum_{j \geq E} (-1)^j H_j(t) = \frac {\eta_{c,E}^R(M,
    N)(t)}{e_R(t)(1 - t)^c} 
  \quad {\text{and}} \quad 
  \eta_{c, E}^R(M,N)(1) = 2^c \cdot c! \cdot \eta_c^R(M, N),
\end{equation} 
where $e_R(t)$ is the multiplicity polynomial of $R$ defined in
\eqref{e_R(t)}.
\end{lemma}

\begin{proof}
  Let $s_0 = 1$, $s_1 = t^{d_1} + \cdots +
  t^{d_c}$, \ldots, and $s_c = t^{d_1 + \cdots + d_c}$; that is, the
  $s_k$ are elementary symmetric functions on the symbols $t^{d_l}$.
  The exactness of \eqref{GradedKoszul} for $j \gg 0$ gives the
  relation
\begin{equation*}
 s_0 H_{j}- s_1 H_{j-2} + s_2 H_{j-4}   + \cdots + (-1)^c s_cH_{j-2c} 
 = 0,
\end{equation*} 
from which it follows that, for $F \gg 0$, we have 
\begin{equation*} \label{polys1}
(1 - s_1x +s_2x^2 - \cdots + (-1)^cs_c x^c) G_F(x,t) =
 b_{0,F}(t) +b_{1,F}(t)x + \cdots + b_{c-1,F}(t) x^{c-1},
\end{equation*}
for some polynomials $b_{i,F}(t)$.  Set $b_F(x,t)$ equal to the right
hand side of this equation and use $\sum_{k=0}^c (-1)^k s_k x^k =
\prod_{l=1}^c (1 - t^{d_l}x)$. Then \eqref{Gasratfun} follows easily.

To establish \eqref{etacrelation}, observe that for $F \gg 0$, we have
$$
\frac{b_F(x,1) }{(1-x)^c} =
G_F(x,1) = \sum_{j \geq 0} \dim_{\sk} \Tor_{F + 2j}^R(M,N) x^j.
$$
Taking $E \gg 0$ to be even, set $\eta_{c, E}^R(M,N)(t)$ to be $b_E(1,
t) - b_{E+1}(1, t)$.  The first equation in \eqref{etacrelation}
follows immediately from \eqref{e_R(t)}, \eqref{alternating}, and
\eqref{Gasratfun}.

The leading coefficients of $P_{\ev}(M,N)$ and $P_{\odd}(M,N)$ are
$b_{E}(1, 1)/(c - 1)!$ and $b_{E+1}(1, 1)/(c - 1)!$, and so the value
of $\eta_{c, E}^R(M,N)(t)$ at $t = 1$ is $2^c \cdot c! \cdot
\eta_c^R(M,N)$.
\end{proof}

\begin{remark} 
  The polynomials $\eta_{c,E}^R(M,N)(t)$ depend on $E$, but, as Lemma
  \ref{polynomials} shows, they have a common value at $t=1$.
\end{remark}

\section{The vanishing of $\eta_c^R$}

Throughout the remainder of this paper, we use the following
notations and assumptions: 
\begin{equation} \label{assumptions}
  \begin{gathered}
    \begin{tabular}{l@{\hspace*{1em}}l}
      $\bullet$ & $\sk$ is a field; \\
      \begin{minipage}{.02\textwidth} $\bullet$ \linebreak \phantom{X} 
      \end{minipage}  
            & \begin{minipage}{.8\textwidth} 
              $R = \sk[x_0, \dots, x_{n+c-1}]/(f_1, \dots, f_c),$ 
              where %
              $\deg{x_i} = 1$ for all $i$ and the $f_l$'s are
              homogeneous polynomials in $\fm = (x_0, \dots,
              x_{n+c-1})$ with $d_l = \deg(f_l)$; 
              \end{minipage}  \\
       $\bullet$ &$c > 0$ and $f_1, \dots, f_c$  forms a regular
       sequence; \\      
       $\bullet$ & $X = \Proj{R}$ is a smooth $\sk$-variety. \\
    \end{tabular}
  \end{gathered}
\end{equation}

\begin{remark} %
  Recall that the variety $X$ is smooth if and only if $\fm$ is the
  radical of the homogeneous ideal generated by the $f_l$'s and the
  maximal minors of the Jacobian matrix $(\partial f_l/\partial x_i)$.
  In particular, by the smoothness assumption, $\fm = (x_0, \dots,
  x_{n+c-1})$ is the only non-regular prime of $R$.  
\end{remark}

For a quasi-projective scheme $Z$ over a field $\sk$, we let $G(Z)$
and $K(Z)$ denote the Grothendieck groups of coherent sheaves and
locally free coherent sheaves, respectively. Recall that if $Z$ is
regular (for example, if it is smooth over $\sk$), then the canonical
map $K(Z) \to G(Z)$ is an isomorphism.  For further explanation and
discussion of the relevant groups and maps, see \cite[\S 2.1]{MPSW}.

From the assumptions (\ref{assumptions}), the smooth variety $X =
\Proj{R} \subset \bP^{n+c-1}$ has dimension $n - 1$ and degree $d=d_1
\cdots d_c$.  When $\sk$ is infinite, there is a linear rational map
$\xymatrix{\bP^{n+c-1} \ar@{-->}[r] & \bP^{n - 1}}$ that determines a
regular function on an open subset containing $X$, and hence it
induces a regular map
\begin{equation*}
 \rho \colon X \to \bP^{n-1}
\end{equation*}
that is finite, flat, and of degree $d$.  As $X$ and $\bP^{n-1}$ are
smooth and $\rho$ is finite, we have the following map and
isomorphisms:
\begin{equation*}
 \rho_* \colon \K{X} \iso G(X) \to G(\bP^{n-1}) \cong
 \K{\bP^{n-1}}.
\end{equation*}
We also have the isomorphism
$$
\bZ[t]/(1 - t)^n \cong
 \K{\bP^{n-1}}
$$
that sends $t \mapsto [\cO(-1)]$. We will often identify $\K{\bP^{n-1}}$
with $\bZ[t]/(1 - t)^n$; for example, if
$\alpha \in \K{X}$, then $\rho_*(\alpha)$ is interpreted as belonging
to $\bZ[t]/(1-t)^n$. Likewise, we identify $\K{\bP^{n-1}} \otimes_\bZ
\bQ$ with $\bQ[t]/(1 - t)^n$.

The following three results are similar to \cite[4.1, 4.2, and
4.3]{MPSW}.  We will establish the vanishing of $\eta^R_c(M, N)$ for
$c > 1$ by employing them in the same way that we used our earlier
results to show the vanishing of $\theta^R(M, N) = 2 \eta^R_1(M, N)$
when $\dim{R}$ is even \cite[Theorem 3.2]{MPSW}.

\begin{lemma} \label{carry over} \cite[Lemma 4.1]{MPSW} %
  Under the assumptions in (\ref{assumptions}) with $\sk$ infinite,
  let $M$ be a finitely generated graded $R$-module and
  $\widetilde{M}$ the associated coherent sheaf on $X$.  Then
\begin{equation*}
 \rho_*([\widetilde{M}]) = (1 - t)^n H_M(t)
\end{equation*}
in $\K{\bP^{n-1}}_\bQ = \bQ[t]/(1 - t)^n$.  In particular, 
\begin{equation*}
 \rho_*(1) = \rho_*([\cO_X]) = e_R(t) = \prod_{i=1}^{c}(1 + t + t^2 +
 \cdots + t^{d_i-1}) \in \bQ[t]/(1 - t)^n.  
\end{equation*}
\end{lemma}

\begin{proof} The proof of \cite[Lemma 4.1]{MPSW}, which is the $c=1$
  case, applies verbatim.
\end{proof}

\begin{lemma} \cite[Lemma 4.2]{MPSW} \label{familiar eqns} %
  Under the assumptions in (\ref{assumptions}) with $\sk$ infinite,
  let $M$ and $N$ be finitely generated graded $R$-modules.  For any
  sufficiently large even integer $E$ and for any integer $m \geq 0$,
  the rational function
$$
(1 - t)^{n + m - c} \frac{\eta_{c,E}^R(M,N)(t)}{(e_R(t))^2}
$$ 
does not have a pole at $t=1$. Its image in $\bQ[t]_{(t)}/(1-t)^n =
\bQ[t]/(1-t)^n = \K{\bP^{n-1}}_\bQ$ satisfies the equation
\begin{equation*}
 (1 - t)^{n + m - c} \frac {\eta_{c,E}^R(M,N)(t)}{(e_R(t))^2}
 =
 (1 - t)^m \left(
           \frac{\rho_*([\widetilde{M}])}{\rho_*(1)} \cdot 
           \frac{\rho_*([\widetilde{N}])}{\rho_*(1)}
           -
           \frac{\rho_*([\widetilde{M}]\cdot [\widetilde{N}])}{\rho_*(1)}
           \right).
\end{equation*} 
In particular, taking $m=c-1$ yields
\begin{equation*}
 (1 - t)^{n - 1} {\eta_c^R(M,N)}
 =
 \frac{(1 - t)^{c-1}}{2^c \cdot c!} \left(
           \frac{d \cdot \rho_*([\widetilde{M}])}{\rho_*(1)} \cdot 
           \frac{d \cdot \rho_*([\widetilde{N}])}{\rho_*(1)}
           -
           \frac{d^2 \cdot \rho_*([\widetilde{M}]\cdot
             [\widetilde{N}])}{\rho_*(1)} 
           \right),
\end{equation*}
where $\deg X = d = d_1 \cdots d_c$. 
\end{lemma}   

\renewcommand{\Hs}[2][]{\operatorname{H}_{#1}}

\begin{proof}
  As in \cite[Lemma 4.2]{MPSW}, start with the equation of Hilbert
  series from \cite[Lemma 7]{AB}, namely
\begin{equation*}
  \frac {H_M H_N}{H_R} = \sum_{j \geq 0} (-1)^j H_j.  
\end{equation*}
Splitting the sum at $E \gg 0$ and using the first relation in
\eqref{etacrelation}  gives
\begin{equation*}
 \frac {H_M \cdot H_N}{H_R} =
 \sum_{j=0}^{E-1} (-1)^j H_j + \frac{\eta_{c,E}^R(M,N)(t)}{e_R(t)
   \cdot (1 - t)^c}.    
\end{equation*}
Upon multiplying  by $(1-t)^m/H_R = (1 - t)^{n+m}/e_R(t)$
and rearranging, this yields
\begin{equation} \label{obstruction eqn}
 \begin{gathered}
  (1-t)^m \frac{(1 - t)^n H_M}{e_R(t)} \cdot \frac{(1 - t)^n H_N}{e_R(t)}
  - 
  (1-t)^m \sum_{j=0}^{E-1}(-1)^j \frac{(1-t)^n H_j}{e_R(t)} 
  \\
  = \frac{\eta_{c, E}^R(M,N)(t)}{(e_R(t))^2} (1-t)^{n+m-c}.   
 \end{gathered}
\end{equation}

The first assertion follows from the fact that the expression before
the equality in (\ref{obstruction eqn}) does not have a pole at $t=1$,
using \eqref{neweq}.  Since both sides of this equation are power
series in powers of $1-t$, we may take their images in ${\bQ[t]/(1 -
  t)^n}$.  Apply Lemma \ref{carry over}.  The expression before the
equality in \eqref{obstruction eqn} becomes
\begin{equation*}
  (1-t)^m \frac{\rho_*([\widetilde{M}])}{\rho_*(1)}
  \frac{\rho_*([\widetilde{N}])}{\rho_*(1)} 
  -
  (1-t)^m \sum_{j=0}^{E-1} (-1)^j 
  \frac{\rho_*([\widetilde{\Tor_j^R(M, N)}])}{\rho_*(1)}.
\end{equation*} 
If $E$ is large enough so that $\Tor_j^R(M,N)$ has finite length for
$j \geq E$, then the alternating sum is $\rho_*([\widetilde{M}] \cdot
[\widetilde{N}])/\rho_*(1)$ where $[\widetilde{M}] \cdot
[\widetilde{N}]$ is multiplication in the ring $K(X)$.  This gives the
first equation in the Lemma.

For the second equation, set $m = c - 1$.  Define $g(t) = \eta_{c,
  E}^R(M, N)(t)/(e_R(t))^2$.  As $e_R(1) = d$ and $\eta_{c,
  E}^R(M,N)(t)$ is a polynomial, $g(t)$ is a rational function without
a pole at $t = 1$.  Thus, modulo $(1 - t)^n$,
\begin{equation*}
 g(t) (1 - t)^{n-1} = g(1)(1 - t)^{n-1} + 
 \frac{g(t) - g(1)}{1 - t} (1 - t)^n \equiv g(1) (1 - t)^{n-1}. 
\end{equation*} 
Multiplication by $d^2$ establishes the second equation.  
\end{proof}

\renewcommand{\Hs}[2][]{\operatorname{H}_{#1}(#2)}

In the next lemma and in the proof of our main theorem, we use \'etale
cohomology.  Assume $\sk$ is a separably closed field, fix a prime
$\ell \neq$ char $\sk$, and write $\H^j_{\et}(Z, \bQ_\ell(i))$ for the
\'etale cohomology of a scheme $Z$ with coefficients in $\bQ_\ell(i)$.
In addition, write $\H^{2*}_{\et}(Z, \bQ_\ell(*))$ for $\bigoplus_i
\H^{2i}_{\et}(Z, \bQ_\ell(i))$.  This is a commutative ring under cup product.
Moreover, the \'etale Chern character gives a ring
homomorphism
\begin{equation*}
  \ch_{\et} \colon \K{Z}_\bQ \to \H_{\et}^{2*}(Z, \bQ_{\ell}(*)).
\end{equation*}
We refer the reader to \cite{FK} for additional background.

\begin{lemma} \cite[Lemma 4.3]{MPSW} \label{familiar diag} %
  Under the assumptions in (\ref{assumptions}) with $\sk$ separably
  closed, the diagram
\begin{equation} \label{ECH}
 \begin{gathered}
 \xymatrix{
 \K{X}_\bQ \ar[rr]^-{\ds \frac{d}{\rho_*(1)} \cdot \rho_*} 
        \ar[d]_-{\ds \ch_{\et}} 
  & & \K{\bP^{n-1}}_\bQ \ar[d]^-{\ds \ch_{\et}} \\
 \H_{\et}^{2*}(X, \bQ_\ell(*))
  \ar[rr]_-{\ds \rho_*^{\et}} 
  & & \H_{\et}^{2*}(\bP^{n-1}, \bQ_\ell(*))
  }
 \end{gathered}
\end{equation}
commutes, where $\rho_*^{\et}$ is the push-forward map for \'etale
cohomology.
\end{lemma}

\begin{proof} The proof of \cite[Lemma 4.4]{MPSW}, which is the $c=1$
  case, applies verbatim.
\end{proof}

The following is the main result of this paper.

\begin{theorem} \label{etacvanish} %
  Under the assumptions in (\ref{assumptions}) with $\sk$ separably
  closed and infinite, let $M$ and $N$ be finitely generated graded
  $R$-modules.  For $E$ a sufficiently large even integer, $\eta_{c,
    E}^R(M, N)(t)$ has a zero at $t=1$ of order at least $c - 1$.  In
  particular, $\eta_c^R(M, N) = 0$ for $c > 1$.
\end{theorem}

\begin{proof}
  The claim is vacuous if $c \leq 1$, and so assume $c > 1$.  Let
  $\gamma$ be the element $\ch_{\et}(1 - t)$ of
  $\H_{\et}^{2*}(\bP^{n-1}, \bQ_\ell(*))$.

We apply $\ch_{\et}$ to $d^2$ times the first equation in Lemma
  \ref{familiar eqns} and use the commutative diagram in Lemma
  \ref{familiar diag} to obtain
\begin{equation} \label{etacchern}
 \begin{gathered}
   \frac{d^2 \cdot  \ch_{\et}\left((1-t)^{n+m-c} \cdot \eta_{c,E}^R(M,
     N)(t)\right)}{\ch_{\et}((e_R(t))^2)}=
   \\
   \gamma^m\left(\rho_*^{\et}\ch_{\et}([\widetilde{M}]) \cdot
     \rho_*^{\et}\ch_{\et}([\widetilde{N}]) - d \cdot \rho_*
     ^{\et}\Big(\ch_{\et}([\widetilde{M}]) \cdot
     \ch_{\et}([\widetilde{N}]) \Big) \right).
 \end{gathered}
 \end{equation}

For $\alpha, \beta \in \H_{\et}^{2*}(X, \bQ_\ell(*))$, define
\begin{equation*}
  \Psi_m(\alpha, \beta) = \gamma^m \big(\rho_*^{\et}(\alpha) \cdot
  \rho_*^{\et}(\beta) - d \cdot \rho_*^{\et}(\alpha \cdot \beta)
  \big).
\end{equation*}  

We will prove that $\Psi_m$ vanishes for any $m \geq 1$. Using the
projection formula 
\begin{equation*}
  \rho_*^{\et}(\rho^*_{\et}(\alpha') \cdot \omega)=\alpha' \rho_*^{\et}(\omega)
\end{equation*} 
with $\omega = 1$, and using the fact that $\rho_*^{\et}(1) = d$, we
see that if $\alpha = \rho^*_{\et}(\alpha')$ for some $\alpha' \in
\H_{\et}^{2*}(\bP^{n-1}, \bQ_\ell(*))$, then
\begin{equation*}
  \Psi_m(\alpha, \beta) = 
  \gamma^m (\alpha' \rho_*^{\et}(1) \rho_*^{\et}(\beta) - d  \cdot
  \alpha' \rho_*^{\et}(\beta)) = 0.
\end{equation*}
Likewise, $\Psi_m(\alpha, \beta) = 0$ if $\beta =
\rho^*_{\et}(\beta')$. But since $X$ is a complete intersection in
projective space, the map
\begin{equation*}
\rho^*_{\et}: \H_{\et}^{2j}(\bP^{n-1}, \bQ_\ell(j)) \to \H_{\et}^{2j}(X, \bQ_\ell(j))
\end{equation*} 
is an isomorphism in all degrees except possibly in degree $2j = n-1$
(see \cite[XI.1.6]{SGA7}).  So we may assume $n$ is odd and that
$\alpha, \beta \in \Het^{n-1}(X, \bQ_\ell(\frac{n-1}{2}))$. Noticing
that $\gamma$ is in $\bigoplus_{i \geq 1} \H_{\et}^{2i}(X,
\bQ_\ell(i))$, we see that $\gamma^m \rho_*^{\et}(\alpha) \cdot
\rho_*^{\et}(\beta)$ and $\gamma^m d \cdot \rho_*^{\et}(\alpha \cdot
\beta)$ belong to
\begin{equation*}
\bigoplus_{j \geq 0} \H_{\et}^{2n-2+2m+2j}(X, \bQ_\ell(n-1+m+j)).
\end{equation*}
This group vanishes when $m \geq 1$ because $\dim(X)=n-1$ and \'etale
cohomology vanishes in degrees more than twice the dimension of a
smooth variety over a separably closed field \cite[X.4.3]{SGA4}.  As
$\Psi_m$ is zero for $m \geq 1$, so too is the expression on the
left-hand side of \eqref{etacchern}.

We have proven that for $m \geq 1$, 
$$
\frac{d^2 \cdot \ch_{\et}((1-t)^{n+m-c} \cdot \eta_{c,E}^R(M,N)(t))}
{\ch_{\et}(e_R(t))^2} = 0
$$
and hence 
$$
 \ch_{\et}((1-t)^{n+m-c} \cdot \eta_{c,E}^R(M,N)(t))  = 0.
$$
But the Chern character map with $\bQ_\ell$ coefficients induces an
isomorphism on projective space,
$$
\ch_{\et}: K(\bP^{n-1}) \otimes \bQ_\ell \xra{\cong}
\H_{\et}^{2*}(\bP^{n-1}, \bQ_\ell(*)), 
$$
and therefore, $(1 - t)^{n + m - c} \cdot \eta_{c, E}^R(M, N)(t) = 0$
in $\bQ[t]/(1 - t)^n$. That is, $(1 - t)^n$ divides $(1 - t)^{n + m -
  c} \cdot \eta_{c, E}^R(M, N)(t)$ in $\bQ[t]$ so that $\eta_{c,
  E}^R(M, N)(t) /(1-t)^{c-m}$ is a polynomial.  Taking $m = 1 <c$
proves the Theorem.
\end{proof}

The familiar example below shows that $\eta_1^R(M, N)$ need not
vanish, even under the assumptions (\ref{assumptions}), if $\dim R$ is
odd.

\begin{example} \cite[Example 1.5]{Ho} \label{exnonzero} %
  The assumption that $c > 1$ is necessary in the Theorem. Let $R =
  \bC[x,y,u,v]/(xu + yv)$ and set $M = R/(x, y)$, $N = R/(u, v)$, and
  $L = R/(x, v)$.  Then $\eta_1^R(M, M) = \frac{1}{2}$, $\eta_1^R(M,
  N) = \frac{1}{2}$, and $\eta_1^R(M, L) = -\frac{1}{2}$.
\end{example}

\begin{corollary} \label{nograding} %
  Under the assumptions in (\ref{assumptions}) and for every pair of
  finitely generated, but not necessarily graded, $R$-modules $M$ and
  $N$, if $\dim R > 0$, then $\eta_c^R(M,N)$ vanishes when $c > 1$.
  When $\dim R = 0$, then $\eta_c^R(M,N)$ vanishes for all $c$.
  \end{corollary}

\begin{proof} 
  Upon passing to any faithfully flat field extension $\sk'$ of $\sk$,
  the assumptions \eqref{assumptions} remain valid, and, moreover, for
  finitely generated $R$-modules $M$ and $N$, the value of
  $\eta_c^{R}(M,N)$ is unchanged.  In more detail, since the lengths
  involved are dimensions over the field $\sk$ for $\eta_c^R$ and
  dimensions over the field $\sk'$ for $\eta_c^{R \otimes_\sk \sk'}$,
  we have equality
\begin{equation*} 
  \eta_c^{R}(M,N) = \eta_c^{R \otimes_\sk \sk'}(M \otimes_\sk
  \sk', N \otimes_\sk \sk').
\end{equation*}
In particular, by passing to the separable closure of $\sk$, we may
assume that $\sk$ is separably closed.

Since $\eta_c^R(-,-)$ is biadditive \cite[Theorem 4.3]{Dao3} and
defined for all pairs of finitely generated $R$-modules, it follows
that $\eta_c^R$ determines a bilinear pairing on $G(R)$, and hence on
$\G{R}_{\bQ}:=\G{R} \otimes_\bZ \bQ$.  It suffices to prove that this
latter pairing is zero.

Assume $\dim R > 0$.  From \cite[Section 2.1]{MPSW}, we recall the
mapping from $\K{X}_{\bQ}$ to $\G{R}_{\bQ}$ given as follows: if $T$
is a finitely generated graded $R$-module with associated coherent
sheaf $\widetilde{T}$ on $X$, then $\K{X}_{\bQ} \to \G{R}_{\bQ}$ sends
$\smash{[\widetilde{T}] \in \K{X}_\bQ}$ to $[T] \in \G{R}_{\bQ}$.  As
proven in \cite[(2.4)]{MPSW}, this mapping is onto, and hence the
vector space $\G{R}_{\bQ}$ is spanned by classes of {\em graded}
$R$-modules.  Therefore, Theorem \ref{etacvanish} applies to prove the
pairing on $\G{R}_{\bQ}$ induced by $\eta_c^R$ is the zero pairing.

Finally, if $\dim R = 0$, then $[R] = \dim_{\sk}(R) \cdot [\sf k]$ in
$\G{R}_\bQ$, and hence $[M] = \length(M) \cdot [\sf k]$ in
$\G{R}_\bQ$.  Since $\eta_c^R(R,R) = 0$ as $R$ is projective, it
follows that $\eta_c^R$ vanishes for all pairs.
\end{proof} 

\begin{corollary} \label{Daostuff} %
  With the assumptions in (\ref{assumptions}), let $M$ and $N$ be
  finitely generated, but not necessarily graded, $R$-modules.  Then
  for $c > 1$, the pair $(M,N)$ is $c$-$\Tor$-rigid; that is, if $c$
  consecutive torsion modules $\Tor_i^R(M, N), \ldots,$
  $\Tor_{i+c-1}^R(M,N)$ all vanish for some $i \geq 0$, then
  $\Tor_j^R(M,N) = 0$ for $j \geq i$.
\end{corollary}

\begin{proof} 
  By Corollary \ref{nograding}, $\eta_c^R(M,N)=0$ when $c > 1$, and
  the result immediately follows from \cite[Theorem 6.3]{Dao3}.
\end{proof}

\appendix
\section{Adapting Gulliksen's Work} \label{App}

\newcommand{\ul}[1]{\underline{#1}}
\newcommand{\Mod}[1]{\operatorname{Mod}_{#1}}
\newcommand{\tensor}{\otimes}
\newcommand{\into}{\lhook\joinrel\longrightarrow}
\newcommand{\onto}{\relbar\joinrel\twoheadrightarrow}
\newcommand{\bN}{\mathbb{N}}

We show in this appendix how to modify Gulliksen's work in \cite{Gu}
to give an alternative proof of Proposition \ref{P's} from the body of
this paper.  The key result is Proposition \ref{AppProp} below, which
was originally proven by Dao \cite[Lemma 3.2]{Dao3}.  Using this
result, a standard argument easily establishes Proposition \ref{P's}.

\begin{prop} \label{AppProp} %
  Assume $Q$ is a noetherian ring, $f_1, \dots, f_c \in Q$ is a
  regular sequence, $R = Q/(f_1, \cdots, f_c)$ and $M$ and $N$ are
  finitely generated $R$-modules. If for all $i \gg 0$,
  $\Tor_i^R(M,N)$ has finite length and $\Tor_i^Q(M,N) = 0$, then
  there exists an integer $j$ such that $\bigoplus_{i \geq j}
  \Tor_i^R(M,N)$ is artinian as a module over the polynomial ring
  $R[\chi_1, \dots, \chi_c]$, where the $\chi_i$'s act via the
  Eisenbud operators.
\end{prop}

Our proof of this result follows Gulliksen's proof of \cite[Theorem
3.1]{Gu}; we use two lemmas, both of which are analogues of his
results.  Our Lemma \ref{G(1.2)} sidesteps the issue raised in
\cite[Example 2.9]{Dao3} while following \cite[Lemma 1.2]{Gu}.  First
we give some notation.  

\begin{chunk} \label{gradedNotation} %
Let $G$ be a $\bZ$ graded ring concentrated in non-positive degrees
(i.e., $G_i = 0$ for all $i > 0$). Note that given a graded $G$-module
$H$, for any integer $r$, we have that $H_{< r} := \bigoplus_{i < r}
H_i$ is a $G$-graded submodule of $H$ and $H_{\geq r} := H/H_{< r}$ is
a $G$-graded quotient module of $H$.  Following Dao, we say that a
graded $G$-module $H$ is {\em almost artinian} if there is an integer
$r$ such that $H_{\geq r}$ is artinian as a $G$-module.
\end{chunk}

\begin{lemma} \label{G(1.2)} %
  Let $G$ and $H$ be as in paragraph \ref{gradedNotation}.  Assume
  that $H_i$ is an artinian $G_0$-module for all $i \gg 0$.  Let $X
  \colon H \to H$ be a homogeneous $G$-linear map of negative degree.
  If $\ker(X)$ is almost artinian as a $G$-module, then $H$ is almost
  artinian as a $G[X]$-module.
\end{lemma}

\begin{proof}
  Let $X$ have degree $w < 0$.  For each $r$,
  there is a map of degree $w$ on quotient modules given by
  multiplication by $X$:
$$
X_{\geq r}: H_{\geq r} \to H_{\geq r+w}.
$$
Note that $\ker(X_{\geq r}) = \ker(X)_{\geq r}$, where the latter uses the
notation of \ref{gradedNotation}, and hence $\ker(X_{ \geq r})$ is artinian
as a $G$-module for $r \gg 0$, by assumption.

Since $w < 0$, there is a canonical surjection $\pi_{\geq r+w} \colon
H_{\geq r+w} \onto H_{\geq r}$ of graded $G$-modules having degree
$0$. Define $Y_{\geq r} = \pi_{\geq r+w} \circ X_{\geq r}$ so that
$Y_{\geq r}$ is the endomorphism of $H_{\geq r}$ of degree $w$ given by
multiplication by $X$, and we have
the left exact sequence
\begin{equation} \label{AppE} %
  0 \to \ker(X_{\geq r})\to \ker(Y_{\geq r}) \stackrel {\cdot X} \to
  \ker(\pi_{\geq r+w}).
\end{equation}

The module $\ker(\pi_{\geq r+w})$ may be regarded as a $G_0$-module
via restriction of scalars along the inclusion $G_0 \into G$.  As a
$G_0$-module, $\ker(\pi_{\geq r+w})$ is a finite direct sum of $H_i$
for $i = r+w, \ldots, r$.  Hence for $r \gg 0$, it is an artinian
$G_0$-module by our assumption.  So for $r \gg 0$, it follows that
$\ker(\pi_{\geq r+w})$ must also be artinian as a $G$-module. 
Thus for $r \gg 0$, the module $\ker(Y_{\geq r})$ is an
artinian $G$-module, as follows from exact sequence \eqref{AppE}.  

It follows from  \cite[Lemma 1.2]{Gu} that $H_{\geq r}$ is
artinian as a graded $G[Y_{\geq r}]$-module, for $r \gg 0$; that is, 
$H$ is an almost artinian $G[X]$-module.  
\end{proof}

The {\em Koszul algebra} $K$ associated to a regular sequence
$(f_1,\dots,f_c)$ of elements of a commutative ring $Q$ is defined to
be the following DG $Q$-algebra: The underlying graded $Q$-algebra is
the exterior algebra $\bigwedge^*_Q(Q^c)$ on the free $Q$-module
$Q^c$, indexed so that $\bigwedge_Q^j(Q^c)$ lies in homological degree
$j$.  Let $T_1,\dots,T_c$ be the standard basis of $Q^c$. The
differential $\del$ of $K$ is uniquely determined by setting
$\del(T_i) = f_i$ and requiring that it satisfy the Leibniz rule:
$\del(ab) = \del(a)b + (-1)^{\deg a}a\del(b)$.

The Koszul algebra comes equipped with a ring map, called the {\em
  augmentation}, to its degree zero homology, namely $H_0(K) =
Q/(f_1,\dots,f_c) =: R$.  Since $(f_1,\dots,f_c)$ is $Q$-regular, the
augmentation $K \twoheadrightarrow R$ is a quasi-isomorphism, so that
$K$ is a DG algebra resolution of $R$ over $Q$ that is free as a
$Q$-module.
Recall that a {\em DG module} over $K$ is a graded $K$-module $L$ equipped
with a differential $d_L$ of degree minus one so that $d_L(a x) = \del(a) x
+ (-1)^{|a|} a d_L(x)$ holds for all homogeneous elements $a \in K$
and $x \in L$. 
See \cite{Avr} for more background material on DG algebras and DG modules.

\newcommand{\bw}{\begin{smallmatrix}}
\newcommand{\ew}{\end{smallmatrix}}
\newcommand{\bb}{\left[ \bw}
\newcommand{\eb}{\ew \right]}

\begin{chunk} \label{AppendixNotation1} %
  Let $K$ be Koszul algebra over $Q$ associated to a regular sequence
  $f_1, \ldots, f_c$ in $Q$ and let $I = \ker(K \onto R)$ be the
  augmentation ideal.  Let $L$ be a DG module over $K$ that is graded
  free as a module over the graded ring underlying $K$.  Let
  $N$ be a finitely generated $R$-module.  Define $Y$ to be $L
  \otimes_Q R = L/(f_1, \dots, f_c) L$.  For any subset $S = \{i_1,
  \ldots, i_s\}$ of $\{1, \ldots, c\}$, define $I_S = (T_{i_1},
  \ldots, T_{i_s})$ and set $Y^S = Y/I_SY$. In particular,
  $Y^\emptyset = Y$ and $Y^{\{1, \ldots, c\}} = Y/IY = L/IL = L
  \otimes_K R$.  For each $S$, the complex $Y^S$ is a complex of
  $Q$-modules, and in fact of $R$-modules.
\end{chunk}

\begin{chunk} \label{AppendixNotation2}
  Gulliksen shows \cite[p. 176--8]{Gu} that for $i \in S$ there is
  an exact sequence of complexes of $Q$-modules
$$
 0 \to Y^S \stackrel{T_i} \to Y^{S \setminus \{i\}} \to Y^S \to 0
$$  
that is degree-wise split exact.  It follows that
\begin{equation} \label{G(6)} %
   0 \to Y^S \otimes_Q N \to Y^{S \setminus
    \{i\}} \otimes_Q N \to Y^S \otimes_Q N \to 0
\end{equation}
is also exact, giving a long exact sequence in homology.  The boundary
map in this sequence, $\H(Y^S \otimes_Q N) \to \H(Y^S \otimes_Q N)$,
is, up to sign, the action of $X_i$ on $\H(Y^S \otimes_Q N)$ as
defined by Gulliksen.  Thus $X_i$ has degree $-2$, since $T_i$ has
degree $+1$.  Gulliksen proves that these actions commute: $X_i X_j =
X_j X_i$ on $\H(Y^S \otimes_QN)$ when $i, j \in S$.  When $S = \{1,
\dots, c\}$, these actions endow $\H(L/IL \otimes_R N)$ with the
structure of a graded module over the graded ring $R[X_1, \dots, X_c]$
where each $X_i$ has degree $-2$.
\end{chunk}

Our next lemma is similar to \cite[Lemma 3.2(ii)]{Gu}.

\begin{lemma} \label{G(3.2)} %
  With the assumptions in paragraph \ref{AppendixNotation1}, and with
  the $X_i$ from paragraph \ref{AppendixNotation2}, if $\H_i(L/IL
  \otimes_R N)$ is artinian as an $R$-module and $\H_i(L \otimes_Q N)
  = 0$ for $i \gg 0$, then $\H(L/IL \otimes_R N)$ is almost artinian
  as an $R[X_1, \dots, X_c]$-module.
\end{lemma}

\begin{proof}
  We have that for $i
  \gg 0$, the $R$-module $\H_i(Y^{\{1, \dots, c\}} \otimes_R N) =
  \H_i(L/IL \otimes_R N)$ is artinian and $\H_i(Y^\emptyset \otimes_Q
  N) = \H_i(L \otimes_Q N) = 0$.
 
  We first observe that, for every $S \subseteq \{1, \dots, c\}$,
  $\H_i(Y^S \otimes_Q N)$ is also artinian over $R$, for $i \gg
  0$. Indeed, this follows immediately by descending induction on the
  cardinality of $S$ using the long exact sequence in homology
  associated to the exact sequence \eqref{G(6)}.

  For $S = \{i_1, \ldots, i_s\}\subseteq\{1,\dots,c\}$, let $G^S =
  R[X_{i_1}, \ldots, X_{i_s}]$. In particular, $G^\emptyset = R$ and
  $G^{\{1, \ldots, c\}} = R[X_1, \dots, X_c]$.  We prove $\H(Y^S
  \otimes_Q N)$ is an almost artinian $G^S$-module for each $S$, by
  induction on the cardinality of $S$.  When $S = \emptyset$, then
  $Y^S = Y^\emptyset = L$ and, by assumption, $\H(Y^\emptyset \otimes_Q
  N) = \H(L \otimes_Q N)$ is an almost artinian $R$-module.

  Assume that $i \in S$.  The exact sequence \eqref{G(6)} gives an
  exact homology sequence 
  $$ 
  \xymatrix{
  \H(Y^{S \setminus \{i\} } \otimes_Q N) \ar[r]
  & \H(Y^S \otimes_Q N) \ar[r]^{X_i} 
    & \H(Y^S \otimes_Q N).
  }
  $$ 
  By the induction hypothesis, 
$\H(Y^{S \setminus \{i\} } \otimes_Q N)$ is almost artinian as a $G^{S \setminus
    \{i\}}$-module, and since the graded quotient of an almost artinian
  module is almost artinian, we see that $\ker(X_i)$ is almost
  artinian.
Since $\H_q(Y^S \otimes_Q N)$ 
is artinian over $G_0 = R$, for
  $q \gg 0$, as was shown above,
Lemma \ref{G(1.2)}  applies to show  $\H(Y^S \otimes_Q N)$ is an almost
artinian $G^S$-module.
\end{proof}

\begin{proof}[Proof of Proposition \ref{AppProp}] %
  Regard $M$ as a DG module concentrated in degree $0$ via restriction
  of scalars along the augmentation.  Gulliksen shows in \cite[Lemma
  2.4]{Gu} how to construct a DG module $L$ over $K$ and a map $L \to
  M$ of DG modules such that $L$ is free over the graded ring
  underlying $K$ and the map $L \to M$ is a quasi-isomorphism.  Note
  that $L \to M$ is, in particular, a resolution of $M$ by free
  $Q$-modules.  Moreover, Gulliksen shows \cite[Lemma 2.6]{Gu} that
  the projection map $L \twoheadrightarrow L/IL = L \otimes_K R$ is a
  quasi-isomorphism where $I = \ker(K\twoheadrightarrow R)$ is the
  augmentation ideal.  In particular, this means that $L/IL$ is an
  $R$-free resolution of $M$.  We therefore obtain the isomorphisms
$$
H_i(L \otimes_Q N) \cong \Tor_i^Q(M,N) \quad \text{  and } \quad
H_i(L/IL \otimes_R N) \cong \Tor_i^R(M,N)
$$
for any $R$-module $N$.

Lemma \ref{G(3.2)} now applies to prove that $\bigoplus_i
\Tor_i^R(M,N)$ is almost artinian as a $R[X_1, \dots,
X_c]$-module. Finally, Avramov and Sun \cite[\S 4]{AvrSun} prove that
the $X_i$'s as constructed by Gulliksen agree with the Eisenbud
operators, up to a sign.
\end{proof}

\begin{ack} 
  We thank the referee for helpful comments regarding this paper.
\end{ack}


\end{document}